\documentclass[12pt]{article}

\usepackage{amssymb}
\usepackage{amsthm}
\usepackage{amsmath}
\usepackage{stmaryrd}
\usepackage{titletoc}
\usepackage{mathrsfs}

\newlength{\defbaselineskip}
\newcommand{\setlinespacing}[1]%
           {\setlength{\baselineskip}{#1 \defbaselineskip}}

\theoremstyle{plain}
\newtheorem{thm}{Theorem}[section]

\newtheorem{lem}[thm]{Lemma}
\newtheorem{prop}[thm]{Proposition}
\newtheorem{exam}[thm]{Example}

\theoremstyle{definition}

\makeatletter\@addtoreset{equation}{section} \makeatother
\begin{document}

\title {Mean-field type Quadratic BSDEs}

\author{H\'el\`ene Hibon\thanks{IRMAR,
Universit\'e Rennes 1, Campus de Beaulieu, 35042 Rennes Cedex, France。 } 
\and
Ying Hu\thanks{IRMAR,
Universit\'e Rennes 1, Campus de Beaulieu, 35042 Rennes Cedex, France (ying.hu@univ-rennes1.fr) 
and School of
Mathematical Sciences, Fudan University, Shanghai 200433, China.
Partially supported by Lebesgue center of mathematics ``Investissements d'avenir"
program - ANR-11-LABX-0020-01,  by ANR CAESARS - ANR-15-CE05-0024 and by ANR MFG - ANR-16-CE40-0015-01.} 
\and
Shanjian Tang\thanks{Department of Finance and Control Sciences, School of
Mathematical Sciences, Fudan University, Shanghai 200433, China (e-mail: sjtang@fudan.edu.cn).
Partially supported by National Science Foundation of China (Grant No. 11631004)
and Science and Technology Commission of Shanghai Municipality (Grant No. 14XD1400400).   }}
\maketitle

% -----------------------------------------------------------------------
\begin{abstract}
In this paper, we give several new results on solvability of a quadratic BSDE whose generator depends also on the mean of both variables. First, we consider such a BSDE using John-Nirenberg's inequality for BMO martingales to estimate its contribution to the evolution of the first unknown variable. Then we consider
the BSDE having an additive expected value of a quadratic generator in addition to the usual quadratic one. In this case, we use a deterministic shift transformation to the first unknown variable, when the usual quadratic generator depends neither on the first variable nor its mean, the general case can be treated by a fixed point argument.
\end{abstract}

\section{Introduction}

Let $\{W_t:=(W_t^1, \ldots, W_t^d)^*, 0\le t\le T\}$ be a
$d$-dimensional standard Brownian motion defined on some
probability space $(\Omega, \mathscr{F}, \mathbb P)$. Denote by $\{\mathscr{F}_t,0\le t\le T\}$ the augmented natural filtration of the
standard Brownian motion $W$.

In this paper, we study the existence and uniqueness of an adapted solution of  the following BSDE:
\begin{equation}\label{new qbsde}
Y_t=\xi+\int_t^T f(s,Y_s,\mathbb E[Y_s],  Z_s, \mathbb E[Z_s])\, ds-\int_t^TZ_s\cdot dW_s, \quad t\in [0,T].
\end{equation}

When $f$ does not depend on $(\bar{y}, \bar{z})$, BSDE~(\ref{new qbsde}) is the classical one, and it is extensively studied in the literature, see the pioneer work of Bismut \cite{Bismut1973, Bismut1976} as well as Pardoux and Peng \cite{PardouxPeng1990}.
When $f(t,y,\bar{y},z, \bar{z})$ is scalar valued and quadratic in $z$ while it does not depend on $(\bar{y},\bar{z})$, BSDE~(\ref{new qbsde}) is the so-called quadratic BSDE and has been  studied by Kobylanski~\cite{Kobylanski2000}, Briand and Hu~\cite{BriandHu2006,BriandHu2008}. BSDE~(\ref{new qbsde}) (called mean-field type BSDE) arises naturally when studying mean-field games, etc. We refer to 
\cite{BuckdahnLiPeng} for the motivation of its study.
When the generator $f$ is uniformly Lipschitz in the last four arguments, BSDE~(\ref{new qbsde}) is shown in a straightforward manner to have a unique adapted solution, and the reader is referred to Buckdahn et al.~\cite{BuckdahnLiPeng} for more details. For the general generator $f(s,y,\bar{y},z,\bar{z})$
depending quadratically on $z$,  BSDE~(\ref{new qbsde}) is a quadratic one  involving both $\mathbb E[Y]$ and $\mathbb E[Z]$. The comparison principle (see \cite{Kobylanski2000}) is well known to play a crucial role in the study  of quadratic BSDEs  (see \cite{Kobylanski2000}). Unfortunately, the comparison principle fails to hold
for BSDE (\ref{new qbsde}) (see , e.g. \cite{BuckdahnLiPeng} for a counter-example for comparison with Lipschitz generators),
the derivation of its solvability is not straightforward. Up to our best knowledge, no study on quadratic mean-field type BSDEs is available.
To tackle the difficulty of lack of comparison princilpe,  we use the John-Nirenberg inequality for BMO martingales to address the solvability.

Furthermore, we study the following alternative of mean-field type BSDE, which admits a quadratic growth in the mean of the second unknown variable $\mathbb E[Z]$:
\begin{eqnarray}
Y_t& = &\xi+\int_t^T\left[f_1(s,Y_s,\mathbb E[Y_s],Z_s,\mathbb E[Z_s])+E[f_2(s,Y_s,\mathbb E[Y_s],Z_s,\mathbb E[Z_s])\right]\, ds\nonumber\\
& & -\int_t^TZ_s\cdot dW_s,\label{bsde:intro}
\end{eqnarray}
where $f_2$ is allowed to grow quadratically  in both $Z$ and $\mathbb E[Z]$, the function $f_1$ also admits a quadratic growth in the second unknown variable for the scalar case.

To deal with the additive expected value of $f_2$, Cheridito and Nam \cite{CheriditoNam2} introduced 
Krasnoselskii fixed point theorem to conclude the existence and uniqueness, by observing that the range of the expected value of $f_2$ is (locally) compact. Here we observe the following fact: the expected value of $f_2$ has no contribution to the second unknown variable $Z$ if $f_1$ depends  neither on the first variable $Z$ nor on its mean. Hence we use the shift transformation to remove the expectation of $f_2$.
In the general case, we apply the same kind of technique and the contraction mapping principle. 

Let us close this section by introducing some notations.
Denote by ${\cal S}^\infty(\mathbb R^n)$ the totality of $\mathbb R^n$-valued $\mathscr{F}_t$-adapted essentially bounded continuous processes, and by $||Y||_\infty$ the essential supremum norm of $Y\in {\cal S}^\infty(R^n)$. It can be verified that $({\cal S}^\infty(\mathbb R^n), ||\cdot||_\infty)$ is a Banach space. Let $M=(M_t, \mathscr{F}_t)$ be a uniformly integrable martingale with $M_0=0$, and for
$p\in [1, \infty)$ we set
\begin{equation}
\begin{split}
 \|M\|_{BMO_p(\mathbb P)}:= \sup_{\tau} \left| \mathbb E_\tau\left[\left(\langle M\rangle_\tau^\infty\right)^{p\over2}\right]^{1\over p}\right|_\infty
\end{split}
\end{equation}
where the supremum is taken over all stopping times $\tau$. The class $\{M : \|M\|_{BMO_p} <
\infty\}$ is denoted by $BMO_p$, which is  written as $BMO_p(\mathbb P)$ whenever it is necessary to indicate the underlying probability, and observe that $\|\cdot\|_{BMO_p}$ is a norm on this space and $BMO_p(\mathbb P)$ is a Banach space.

Denote by $\mathscr{E}(M)$ the stochastic exponential of a one-dimensional local martingale $M$ and by $\mathscr{E}(M)_s^t$ that of $M_t-M_s$. Denote by $\beta \cdot M$ the stochastic integral of a scalar-valued adapted  process $\beta$ with respect to a local continuous martingale $M$.

For any real $p\ge 1$, ${\cal S}^p(\mathbb R^n)$ denotes the set of $\mathbb R^n$-valued adapted and c\`adl\`ag processes $(Y_t)_{t \in [0,T]}$ such that
$$||Y||_{{\cal S}^p(\mathbb R^n)}:=\mathbb{E} \left[\sup_{0 \le t\leq T} |Y_t|^p \right]^{1/p} < + \infty,$$
and ${\cal M}^p(\mathbb R^{d\times n})$ denotes the set of adapted processes $(Z_t)_{t \in [0,T]}$ with values in $\mathbb{R}^{d \times n}$ such that
$$||Z||_{{\cal M}^p(\mathbb{R}^{d \times n})}:=\mathbb{E}\left[\left(\int_0^T |Z_s|^2 ds \right)^{p/2}\right]^{1/p} < +\infty.$$

The rest of our paper is organized as follows. In Section 2, we study BSDE~(\ref{new qbsde}) when $f(t,y,\bar{y},z, \bar{z})$ is scalar valued and quadratic in $z$, and  uniformly Lipschitz in $(y,\bar{y},\bar{z})$, and prove by the contraction mapping principle that BSDE~(\ref{new qbsde}) has a unique solution. In Section 3, we study scalar-valued BSDE~(\ref{bsde:intro}) when $f_2(t,y,\bar{y},z, \bar{z})$ is  both quadratic in $z$ and $\bar{z}$, and $f_1$ is quadratic in $z$.
Finally, in Section 4, we study BSDE~(\ref{bsde:intro}) in the multi-dimentional case, where we suppose that  $f_2(t,y,\bar{y},z, \bar{z})$ is  both quadratic in $z$ and $\bar{z}$, and $f_1$ is Lipschitz in $z$ and $\bar{z}$.

\section{Quadratic BSDEs with a mean term involving the  second unknown variable}

%\subsection{The generator is uniformly Lipschitz and grows sub-quadratically in the mean of the second unknown variable}

In this section we consider the following BSDE:
\begin{equation}\label{new bsde}
Y_t=\xi+\int_t^T f(s,Y_s,\mathbb E[Y_s], Z_s, \mathbb E[Z_s])\, ds-\int_t^TZ_s\cdot dW_s, \quad t\in [0,T]
.
\end{equation}

We first recall the following existence and uniqueness, a priori estimate for one-dimensional BSDEs.

\begin{lem}\label{Quad bsde} Assume that (i) the function $f: \Omega\times [0,T]\times \mathbb R\times  \mathbb R^d\to \mathbb R$  has the following  growth and locally Lipschitz continuity in the last two variables:
\begin{equation}
\begin{split}
|f(s,y, z)|\le&\  g_s+\beta|y|+{\gamma\over2}|z|^2, \quad y\in \mathbb R,  z\in R^d;\\
|f(s,y,z_1)-f(s,y,z_2)|\le & C|y_1-y_2|+C(1+|z_1|+|z_2|)|z_1-z_2|, \quad y_1,y_2\in \mathbb R, z_1,z_2\in  \mathbb R^d;
\end{split}
\end{equation}
(ii) the process $f(\cdot,y, z)$ is $\mathscr{F}_t$-adapted for each $y\in \mathbb R$,  $z\in \mathbb R^d$; and (iii)  $g\in L^1(0,T)$. Then for bounded $\xi$, the following BSDE
\begin{equation}\label{new bsde}
Y_t=\xi+\int_t^T [f(s,Y_s, Z_s)]\, ds-\int_t^TZ_s\cdot dW_s, \quad t\in [0,T]
\end{equation}
has a unique solution $(Y, Z)$ such that $Y$ is (essentially) bounded and $Z\cdot W$ is a BMO martingale. Furthermore, we have
$$
e^{\gamma |Y_t|}\le \mathbb E_t\left[e^{\gamma e^{\beta(T-t)} \xi+\gamma\int_t^T|g(s)|e^{\beta(s-t)}\, ds }\right].
$$
\end{lem}

 The following lemma plays an important role in our subsequent arguments. It indicates that  following the proof of \cite[Theorem 3.3, page 57]{Kazamaki} can give a more precise dependence of the two constants $c_1, c_2$ on $\beta\cdot M$.

\begin{lem} For $K>0$, there are constants $c_1>0$ and $c_2>0$ depending only on $K$ such that for any BMO martingale $M$, we have for any one-dimensional BMO martingale $N$ such that $\|N\|_{BMO_2(\mathbb P)}\le K$,
\begin{equation}\label{uniform BMO estimate}
c_1\|M\|_{BMO_2(\mathbb P)}\le \|\widetilde M\|_{BMO_2(\widetilde {\mathbb P})}\le c_2\|M\|_{BMO_2(\mathbb P)}
\end{equation}
where ${\widetilde M}:=M-\langle M, N\rangle$ and $d\widetilde {\mathbb P}:= \mathscr{E}(N)_0^\infty d \mathbb P$.
\end{lem}

\subsection{Main Results}
We make the following three assumptions. Let $C$ be a positive constant.

\medskip
($\mathscr{A}1$) Assume that there are positive constants $C$ and $\gamma$ and $\alpha\in [0,1)$  such that the function $f: \Omega\times [0,T]\times \mathbb R^2\times (\mathbb R^d)^2\to \mathbb R$  has the following linear-quadratic growth and globally-locally Lipschitz continuity: for $\forall (\omega, s,y_i,\bar y_i, z_i, \bar z_i)\in \Omega\times [0,T]\times \mathbb R^2\times (\mathbb R^d)^2$ with $i=1,2$,
\begin{equation}
\begin{split}
|f(\omega, s,y,\bar y,z, \bar z)|\le&\  C(2+|y|+|\bar y|+|\bar z|^{1+\alpha})+{1\over2}\gamma |z|^2;\\
|f(s,y_1,\bar y_1,z_1, \bar z_1)-f(s,y_2,\bar y_2,z_2, \bar z_2)|
\le& C\biggl\{|y_1-y_2|+|\bar y_1-\bar y_2|+(1+|\bar z_1|^\alpha+|\bar z_2|^\alpha)|\bar z_1-\bar z_2|\\
&+(1+|z_1|+|z_2|)|z_1-z_2|\biggr\};
\end{split}
\end{equation}
The process $f(\cdot,y,\bar y,z,\bar z)$ is $\mathscr{F}_t$-adapted for each $(y,\bar y, z, \bar z)$.

\medskip
($\mathscr{A}2$) The terminal condition $\xi$ is uniformly bounded by $C$.

\medskip

We have the following two theorems.

The first one is a result concerning local solutions. For this, let us introduce some notations. For $\varepsilon>0$, and $r_{\varepsilon}>0$, we define the ball $\mathbb{B}_{\varepsilon}$  by
\begin{eqnarray*}
  \mathbb{B}_{\varepsilon}:=\biggl\{ (Y,Z): & Y\in {\cal S}^\infty, \quad  Z\cdot W\in  BMO_2(\mathbb P),\\
  & ||Y||_{\infty, [T-\varepsilon,T]}+\|Z\cdot W\|_{BMO_2,[T-\varepsilon,T]}\le r_\varepsilon\biggr\}.
\end{eqnarray*}

\begin{thm}\label{thm 1} Let assumptions $(\mathscr{A}1)$ and $(\mathscr{A}2)$ be satisfied with $\alpha\in [0,1)$. Then, for any bounded $\xi$, there exist $\varepsilon>0$ and $r_{\varepsilon}>0$ such that the following BSDE
\begin{eqnarray}
    \begin{split}
Y_t=\, & \xi+\int_t^T f(s,Y_s,\mathbb E[Y_s],Z_s, \mathbb E[Z_s])\, ds-\int_t^T Z_s\cdot dW_s, \quad t\in [0,T]
\end{split}
\end{eqnarray}
has a unique local solution $(Y,Z)$ in the time interval $[T-\varepsilon,T]$ with $(Y,Z)\in \mathbb{B}_{\varepsilon}$.
\end{thm}
 
\begin{exam} The condition on $f$ means that $f$ is of linear growth with respect to $(y,\bar{y})$,
and of $|\bar{z}|^{1+\alpha}$ growth. For example, 
For $\alpha\in (0,1)$,
$$f(s,y,\bar{y},z,\bar{z})=1+|y|+|\bar{y}|+\frac{1}{2}|z|^2+|\bar{z}|^{1+\alpha}.$$
\end{exam}

\medskip

The second theorem is a result about global solutions.
\begin{thm}\label{thm 2} Let assumption $(\mathscr{A}2)$ be satisfied. Moreover, assume that there is a positive constant $C$ such that
the function $f: \Omega\times [0,T]\times \mathbb R^2\times (\mathbb R^d)^2\to \mathbb R$  has the following linear-quadratic growth and globally-locally Lipschitz continuity: for $\forall (\omega, s,y_i,\bar y_i, z_i, \bar z_i)\in \Omega\times [0,T]\times \mathbb R^2\times (\mathbb R^d)^2$ with $i=1,2$,
\begin{equation}\label{xxx}
\begin{split}
&|f(s,0,0,0,0)+h(\omega, s,y,\bar y,z, \bar z)|\le\  C(1+|y|+|\bar y|), \\
&|f(\omega,s,0,0,z_1,0)-f(\omega,s,0,0,z_2,0)|\le \ C(1+|z_1|+|z_2|)|z_1-z_2|, \\
&|h(s,y_1,\bar y_1,z_1, \bar z_1)-h(s,y_2,\bar y_2,z_2, \bar z_2)|
 \le\  C\left(|y_1-y_2|+|\bar y_1-\bar y_2|+|z_1-z_2|+|\bar z_1-\bar z_2|\right)
\end{split}
\end{equation}
where for $(\omega, s,y,\bar y, z, \bar z)\in \Omega\times [0,T]\times \mathbb R^2\times (\mathbb R^d)^2,$
\begin{equation}
    h(s, y, \bar y, z, \bar z):= f(s,y, \bar y, z, \bar z)-f(s,0,0,z,0). 
\end{equation}
The process $f(\cdot,y,\bar y,z,\bar z)$ is $\mathscr{F}_t$-adapted for each $(y,\bar y, z, \bar z)\in \mathbb  R^2\times (\mathbb  R^d)^2$.

Then, the following BSDE
\begin{eqnarray}
    \begin{split}
Y_t=\, & \xi+\int_t^T f(s,Y_s,\mathbb E[Y_s],Z_s, \mathbb E[Z_s])\, ds-\int_t^T Z_s\cdot dW_s, \quad t\in [0,T]
\end{split}
\end{eqnarray}
has a unique adapted solution $(Y, Z)$ on $[0, T]$ such that $Y$ is bounded. Furthermore,  $Z\cdot W$ is a $BMO(\mathbb P)$ martingale.
\end{thm}
\begin{exam} The inequality (\ref{xxx}) requires that $f$ is bounded with respect to the last variable $\bar{z}$. The following function 
$$f(s,y,\bar{y},z,\bar{z})=1+s+|y|+|\bar{y}|+\frac{1}{2}|z|^2+|\sin(\bar{z})|, \quad (s,y,\bar{y},z,\bar{z})\in [0,T]\times \mathbb R^2\times (\mathbb R^d)^2,$$
satisfies such an inequality.
\end{exam}

\subsection{Local solution: the proof of Theorem~\ref{thm 1}}

We prove Theorem~\ref{thm 1} (using the contraction mapping principle) in the following three subsections:
in Subsection~\ref{first step}, we construct a map (which we call quadratic solution map) in a Banach space; in Subsection~\ref{second step}, we show that this map is stable in a small ball; and
in Subsection~\ref{third step}, we prove that this map is a contraction.

\subsubsection{Construction of the map}\label{first step}

For a pair of bounded adapted process $U$ and BMO martingale $V\cdot W$, we consider the following quadratic BSDE:
\begin{equation}\label{qbsde2}
Y_t=\xi+\int_t^T f(s,Y_s, \mathbb E[U_s], Z_s, \mathbb E[V_s])\, ds -\int_t^TZ_s\cdot dW_s, \ t\in [0, T].
\end{equation}
As
$$|f(s,y,\mathbb E[U_s],Z_s,\mathbb E[V_s])|\le C(2+|\mathbb E[U_s]|+|\mathbb E[V_s]|^{1+\alpha})+C|y|+\frac{1}{2}{\gamma}|z|^2,
$$
%where $\tilde{c}=C(2+|E[U_s]|)$, $\gamma=\frac{C}{4}+2C$.
in view of Lemma~\ref{Quad bsde}, it has a unique adapted solution $(Y,Z)$ such that $Y$ is bounded and $Z\cdot W$ is a BMO martingale. Define the quadratic solution map $\Gamma: (U,V)\mapsto \Gamma (U,V)$ as follows:
$$\Gamma(U,V):=(Y,Z), \quad \forall (U, V\cdot W)\in {\cal S}^\infty\times BMO_2(\mathbb P).$$
It is a transformation in the Banach space ${\cal S}^\infty\times BMO_2(\mathbb P)$.

Let us introduce here some constants and a quadratic (algebraic) equation which will be used in the next two subsections.

Define
\begin{equation}\label{C delta}
     C_\delta:=e^{{6\over 1-\alpha}\gamma CTe^{CT}+ {1-\alpha\over 2} \left({3\over 1-\alpha}\gamma Ce^{CT}\right)^{\frac{2}{1-\alpha}}\left({1+\alpha\over 2\delta}\right)^{\frac{1+\alpha}{1-\alpha}}T};
\end{equation}
\begin{equation}\label{beta}
    \beta:={1\over 2}(1-\alpha)C^{2\over 1-\alpha}(2(1+\alpha))^{1+\alpha\over 1-\alpha};
\end{equation}
\begin{eqnarray}\label{mu1,2}
    \begin{split}
\label{lambda}
    \mu_1:=&\ (1-\alpha)\left(1+{1-\alpha\over (1+\alpha)\gamma }\right)= 1-\alpha+{(1-\alpha)^2\over  (1+\alpha)\gamma} ;\\
     \mu_2:=&\ {1\over2}(1+\alpha)\left(1+{1-\alpha\over  (1+\alpha)\gamma }\right)= {1\over2}(1+\alpha)+{1-\alpha\over 2\gamma};
\end{split}
\end{eqnarray}
\begin{equation}\label{mu}
\mu:= \left(\beta+C\mu_1\right)\gamma^{2\over \alpha-1}+2C\mu_2.
\end{equation}
Consider the following standard quadratic equation of $A$:
$$
\delta A^2-\left(1+4\gamma^{-2}e^{\gamma |\xi|_\infty} \delta\right)A +4\gamma^{-2}e^{\gamma |\xi|_\infty} +4\mu C_\delta e^{{3e^{CT}\over 1-\alpha}\gamma |\xi|_\infty} \varepsilon=0.
$$
The discriminant of the quadratic equation reads
\begin{eqnarray}
    \begin{split}
\Delta:=&\left(1+4\gamma^{-2}e^{\gamma |\xi|_\infty} \delta\right)^2-4\delta \left[ 4\gamma^{-2}e^{\gamma |\xi|_\infty} +4\mu C_\delta e^{{3e^{CT}\over 1-\alpha}\gamma |\xi|_\infty} \varepsilon\right]\\
=&\left(1-4\gamma^{-2}e^{\gamma |\xi|_\infty} \delta\right)^2-16\mu \delta C_\delta e^{{3e^{CT}\over 1-\alpha}\gamma |\xi|_\infty} \varepsilon.
\end{split}
\end{eqnarray}

Take
\begin{eqnarray}\label{A}
    \begin{split}
\delta:= &{1\over 8}\gamma^2 e^{-\gamma |\xi|_\infty}, \quad \varepsilon\le  \min\left\{{1\over 3C}e^{-CT}, { 1\over 8\mu C_\delta}\gamma^{-2} e^{\gamma(1-{3e^{CT}\over 1-\alpha}) |\xi|_\infty}\right\},\\
 A:=& {{1+4\gamma^{-2}e^{\gamma |\xi|_\infty} \delta-\sqrt{\Delta}}\over 2 \delta}={3-2\sqrt{\Delta}\over 4\delta }\le {3\over 4\delta}=6\gamma^{-2}e^{\gamma |\xi|_\infty},
\end{split}
\end{eqnarray}
and we have
\begin{eqnarray}\label{root}
    \begin{split}
&\Delta \ge 0, \quad 1-\delta A={1+2\sqrt{\Delta}\over 4}>0, \\
& \gamma^{-2}e^{\gamma |\xi|_\infty}+\mu C_\delta{e^{{3\gamma e^{CT}\over 1-\alpha} |\xi|_\infty}\over 1-\delta A}\varepsilon +{1\over4}A= {1\over2} A.
\end{split}
\end{eqnarray}

Throughout this section, we base our discussion on the time interval $[T-\varepsilon, T].$

We shall prove Theorem~\ref{thm 1} by showing that the quadratic solution map $\Gamma$ is a contraction on the closed convex set $\mathscr{B}_{\varepsilon}$ defined by
\begin{eqnarray}
  \mathscr{B}_{\varepsilon}:=\biggl\{ (U,V):&  U\in {\cal S}^\infty, \quad  V\cdot W\in  BMO_2(\mathbb P), \nonumber \\
   &\|V\cdot W\|^2_{BMO_2}\le A,  \quad
    e^{{2\over 1-\alpha}\gamma |U|_\infty}\le {C_\delta e^{{3\gamma e^{CT}\over 1-\alpha} |\xi|_\infty}\over 1-\delta A}\biggr\}
\end{eqnarray}
(where $(U, V)$ is defined on $\Omega\times [T-\varepsilon, T]$ ) for a positive constant  $\varepsilon$ (to be determined later).

\subsubsection{Estimation of the quadratic solution map} \label{second step}

We shall show the following assertion: $\Gamma(\mathscr{B}_{\varepsilon})\subset \mathscr{B}_{\varepsilon}$, that is,
\begin{equation}\label{desired}
\Gamma(U,V)\in \mathscr{B}_{\varepsilon}, \quad\quad \forall\ (U,V)\in \mathscr{B}_{\varepsilon}.
\end{equation}

{\it Step 1. Exponential transformation.}

\medskip
Define
\begin{equation}
\phi(y):=\gamma^{-2}[\exp{(\gamma |y|)}-\gamma |y|-1], \quad y\in R.
\end{equation}
Then, we have for $y\in R, $
\begin{equation}\label{phi}
\phi'(y)=\gamma^{-1}[\exp{(\gamma |y|)}-1]\mbox{sgn}(y), \quad \phi''(y)=\exp{(\gamma |y|)}, \quad \phi''(y)-\gamma |\phi'(y)|=1.
\end{equation}

Using It\^o's formula, we have for $t\in [T-\varepsilon, T],$
\begin{eqnarray}\label{quadratic estimate}
    \begin{split}
      &\phi(Y_t)+{1\over2} \mathbb E_t\left[\int_t^T|Z_s|^2\, ds\right] \\
      \le\, & \phi(|\xi|_\infty)+ C\mathbb E_t\left[\int_t^T|\phi'(Y_s)|\left(2+|Y_s|+|\mathbb E[U_s]|+|\mathbb E[V_s]|^{1+\alpha}\right)\, ds\right].
    \end{split}
\end{eqnarray}
Since (in view of the definition of notation $\beta$ in~\eqref{beta})
$$
C|\phi'(Y_s)| |\mathbb E[V_s]|^{1+\alpha}\le {1\over4}|\mathbb E[V_s]|^2+\beta |\phi'(Y_s)|^{2\over 1-\alpha},
$$
we have
\begin{eqnarray}\label{Young}
    \begin{split}
      &\phi(Y_t)+{1\over2} \mathbb E_t\left[\int_t^T|Z_s|^2\, ds\right] \\
        \le\, & \phi(|\xi |_\infty)+C \mathbb E_t\left[\int_t^T|\phi'(Y_s)|\left(2+|Y_s|+|\mathbb E[U_s]|\right)\, ds\right]\\
                &+\beta \mathbb E_t\left[\int_t^T|\phi'(Y_s)|^{{2\over 1-\alpha}}\, ds\right]+{1\over4}E_t\left[\int_t^T|E[V_s]|^2\, ds\right]\\
                 \le\, & \phi(|\xi|_\infty)+C E_t\left[\int_t^T|\phi'(Y_s)|\left((1+|Y_s|)+(1+|\mathbb E[U_s]|)\right)\, ds\right]\\
        &+\beta \mathbb E_t\left[\int_t^T|\phi'(Y_s)|^{{2\over 1-\alpha}}\, ds\right]+{1\over4}\mathbb E_t\left[\int_t^T|\mathbb E[V_s]|^2\, ds\right].
   \end{split}
\end{eqnarray}
In view of the inequality for $x>0,$
$$
1+x\le \left(1+{1-\alpha\over \gamma(1+\alpha)}\right)e^{{\gamma(1+\alpha)\over 1-\alpha}x},
$$
we have
\begin{eqnarray}
    \begin{split}
&C \mathbb E_t\left[\int_t^T|\phi'(Y_s)|\left((1+|Y_s|)+(1+|\mathbb E[U_s]|)\right)\, ds\right] \\
\le & C \mathbb E_t\left[\int_t^T|\phi'(Y_s)|\left(1+{1-\alpha\over \gamma (1+\alpha)}\right)\left(e^{\gamma {1+\alpha\over 1-\alpha}|Y_s|}+e^{\gamma {1+\alpha\over 1-\alpha}|\mathbb E[U_s]|}\right)\, ds\right].
\end{split}
\end{eqnarray}
Since (by Young's inequality)
$$
|\phi'(Y_s)|\left(e^{\gamma {1+\alpha\over 1-\alpha}|Y_s|}+e^{\gamma {1+\alpha\over 1-\alpha}|E[U_s]|}\right)\le (1-\alpha)|\phi'(Y_s)|^{2\over 1-\alpha}+{1+\alpha\over2}\left(e^{{2\over 1-\alpha}\gamma|Y_s|}+e^{{2 \over 1-\alpha}\gamma|\mathbb E[U_s]|}\right),
$$
in view of the definition of the notations $\mu_1$ and $\mu_2$ in~\eqref{mu1,2}, we have
\begin{eqnarray}
    \begin{split}
&C E_t\left[\int_t^T|\phi'(Y_s)|\left((1+Y_s)+(1+|E[U_s]|)\right)\, ds\right] \\
\le &\  C\mu_1 E_t\left[\int_t^T|\phi'(Y_s)|^{2\over 1-\alpha}\, ds\right]
+ C\mu_2 E_t\left[\int_t^T\left(e^{{2\gamma \over 1-\alpha}|Y_s|}+e^{{2\gamma \over 1-\alpha}|E[U_s]|}\right)\, ds\right].
  \end{split}
\end{eqnarray}

In view of inequality~\eqref{Young}, we have
\begin{eqnarray}\label{fundamental}
    \begin{split}
    &\phi(Y_t)+{1\over2} E_t\left[\int_t^T|Z_s|^2\, ds\right] \\
             \le\, & \phi(|\xi|_\infty)+\left(\beta+C\mu_1\right) E_t\left[\int_t^T|\phi'(Y_s)|^{{2\over 1-\alpha}}\, ds\right]\\
        &+C\mu_2 E_t\left[\int_t^T\left(e^{{2\gamma \over 1-\alpha}|Y_s|}+e^{{2\gamma \over 1-\alpha}|E[U_s]|}\right)\, ds\right]+{1\over4}\int_t^T|E[V_s]|^2\, ds\\
  \le\, & \phi(|\xi|_\infty)+\left[C\mu_2+\gamma^{2\over \alpha-1}\left(\beta+C\mu_1\right)\right] E_t\left[\int_t^T e^{{2\gamma \over 1-\alpha}|Y_s|}\, ds\right]\\
        &+C\mu_2 E_t\left[\int_t^Te^{{2\gamma \over 1-\alpha}|U_s|}\, ds\right]+{1\over4}E\left[\int_t^T|V_s|^2\, ds\right].
    \end{split}
\end{eqnarray}

{\it Step 2. Estimate of $e^{\gamma |Y|_\infty}$.}

\medskip
In view of the last inequality of Lemma~\ref{Quad bsde},  we have
\begin{equation}
   e^{{3\over 1-\alpha}\gamma |Y_t|}\le E_t\left[e^{{3\over 1-\alpha}\gamma e^{CT} \left(|\xi|_\infty+C\int_t^T(2+|E[U_s]|+|E[V_s]|^{1+\alpha})\, ds\right)}\right].
\end{equation}
Since (by Young's inequality)
\begin{equation}
    {3e^{CT}\over 1-\alpha}\gamma C|E[V_s]|^{1+\alpha}
 \le   {1-\alpha\over 2} \left({3e^{CT}\over 1-\alpha}\gamma C \left({1+\alpha\over 2\delta}\right)^{\frac{1+\alpha}{2}}\right)^{2\over 1-\alpha}+\delta |E[V_s]|^2,
\end{equation}
in view of the definition of notation $C_\delta$ in~\eqref{C delta}, we have
\begin{eqnarray}
    \begin{split}
e^{{3\over 1-\alpha} \gamma|Y_t|}\le\, & C_\delta \left[e^{\left({3\over 1-\alpha}\gamma e^{CT}|\xi|_\infty+{3\over 1-\alpha}\gamma Ce^{CT}\varepsilon |U|_\infty +\delta \int_t^T|E[V_s]|^2\, ds\right)}\right];
 \end{split}
\end{eqnarray}
and therefore using Jensen's inequality,
\begin{eqnarray}
    \begin{split}
e^{{3\over 1-\alpha}\gamma |Y_t|}\le\, & C_\delta\ e^{\left({3\over 1-\alpha}\gamma e^{CT}(|\xi|_\infty+ C\varepsilon |U|_\infty)\right)}
E\left[e^{\delta\int_t^T|V_s|^2\, ds}\right].
\end{split}
\end{eqnarray}
It follows from  \eqref{A} and the definition of ${\cal B}_\varepsilon$ that $\|\sqrt{\delta}V\cdot W\|^2_{BMO_2(\mathbb P)}\le \delta A<1$, applying the John-Nirenberg inequality to the BMO martingale $\sqrt{\delta}V\cdot W$, we have
\begin{eqnarray}
    \begin{split}
e^{{3\over 1-\alpha}\gamma |Y_t|}\le\, & {C_\delta \ e^{({3e^{CT}\over 1-\alpha}\gamma |\xi|_\infty+{3e^{CT}\over 1-\alpha}C\gamma \varepsilon |U|_\infty)}\over 1-\delta \|V\cdot W\|^2_{BMO_2}}\le {C_\delta \ e^{({3e^{CT}\over 1-\alpha}\gamma |\xi|_\infty+{3e^{CT}\over 1-\alpha}C\gamma \varepsilon |U|_\infty)}\over 1-\delta A}.
 \end{split}
\end{eqnarray}
Since $3e^{CT}C\varepsilon\le 1$ (see the choice of $\varepsilon$ in~\eqref{A}) and  $(U,V)\in \mathscr{B}_\varepsilon$, we have
\begin{eqnarray}
    \begin{split}
e^{({3\over 1-\alpha}\gamma |Y|_\infty)}
\le\, &  {C_\delta e^{({3e^{CT}\over 1-\alpha}\gamma |\xi|_\infty+{1\over 1-\alpha}\gamma |U|_\infty)}\over 1-\delta A}\\
\le\, & C_\delta { e^{{3e^{CT}\over 1-\alpha}\gamma |\xi|_\infty}\over 1-\delta A}   \left( C_\delta {e^{{3e^{CT}\over 1-\alpha}\gamma |\xi|_\infty}\over 1-\delta A}\right)^{1\over2}\\
\le\, & \left({ C_\delta e^{{3e^{CT}\over 1-\alpha}\gamma |\xi|_\infty}\over 1-\delta A}\right)^{3\over2},
\end{split}
\end{eqnarray}
which gives a half of the desired result~\eqref{desired}.

\medskip
{\it Step 3. Estimate of $\|Z\cdot W\|_{BMO_2}^2$.}

\medskip
From inequality~\eqref{fundamental} and the definition of notation $\mu$ in (\ref{mu}), we have
\begin{eqnarray}\label{A*}
    \begin{split}
    {1\over2} E_t\left[\int_t^T|Z_s|^2\, ds\right]
  \le\, & \gamma^{-2}e^{\gamma |\xi|_\infty}+\mu C_\delta{e^{{3e^{CT}\over 1-\alpha}\gamma |\xi|_\infty}\over 1-\delta A}\varepsilon +{1\over4}A.
\end{split}
\end{eqnarray}
In view of~\eqref{root}, we have
\begin{equation}
   \frac{1}{2} \|Z\cdot W\|_{BMO_2}^2\le \gamma^{-2}e^{\gamma |\xi|_\infty}+\mu C_\delta{e^{{3e^{CT}\over 1-\alpha}\gamma |\xi|_\infty}\over 1-\delta A}\varepsilon +{1\over4}A= {1\over2} A.
\end{equation}
The other half  of the desired result~\eqref{desired} is then proved.

\subsubsection{Contraction of the quadratic solution map}\label{third step}
For $(U,V)\in \mathscr{B}_\varepsilon$ and $(\widetilde U, \widetilde V)\in \mathscr{B}_\varepsilon$, set
$$
(Y,Z):=\Gamma(U,V), \quad (\widetilde Y, \widetilde Z):= \Gamma(\widetilde U, \widetilde V).
$$
That is,
\begin{eqnarray}
    \begin{split}
Y_t=\, & \xi+\int_t^Tf(s,Y_s,E[U_s], Z_s, E[V_s])\, ds -\int_t^TZ_sdW_s,\\
{\widetilde Y}_t=\, & \xi+\int_t^T f(s,{\widetilde Y}_s, E[{\widetilde U}_s], {\widetilde Z}_s, E[{\widetilde V}_s])\, ds -\int_t^T{\widetilde Z}_s dW_s.
\end{split}
\end{eqnarray}
We can define the vector process $\beta$ in an obvious way such that
\begin{eqnarray}
    \begin{split}
&|\beta_s|\le  C(1+|Z_s|+|{\widetilde Z}_s|),\\
&{f(s,Y_s,E[U_s], Z_s, E[V_s])-f(s,Y_s,E[U_s], {\widetilde Z}_s, E[V_s])}=  (Z_s-\widetilde Z_s)\beta_s.
\end{split}
\end{eqnarray}
Then $\widetilde W_t:=W_t-\int_0^t\beta_s\, ds$  is a Brownian motion under the equivalent probability measure ${\widetilde P}$ defined by
$$d{\widetilde P}: =\mathscr{Exp} (\beta\cdot W)_0^T\, dP,$$
 and
from the above-established a priori estimate, there is $K>0$ such that $\|\beta\cdot W\|^2_{BMO_2}\le K^2:=3C^2T+6C^2A$.

In view of the following equation
\begin{eqnarray}
    \begin{split}
&Y_t-{\widetilde Y}_t +\int_t^T(Z_s-{\widetilde Z}_s) \, d{\widetilde W}_s\\
=\, & \int_t^T\left[f(s,Y_s,E[U_s], {\widetilde Z}_s, E[V_s])-f(s,{\widetilde Y}_s,E[{\widetilde U}_s], {\widetilde Z}_s, E[{\widetilde V}_s])\right]\, ds,
\end{split}
\end{eqnarray}
taking square and then the conditional expectation with respect to ${\widetilde P}$ (denoted by ${\widetilde E}_t$) on both sides of the last equation, we have the following standard estimates:
\begin{eqnarray}\label{square}
    \begin{split}
&|Y_t-{\widetilde Y}_t|^2+{\widetilde E}_t\left[\int_t^T|Z_s-{\widetilde Z}_s|^2\, ds\right]\\
= \, & {\widetilde E}_t\left[\left(\int_t^T\left(f(s,Y_s,E[U_s], {\widetilde Z}_s, E[V_s])-f(s,{\widetilde Y}_s,E[{\widetilde U}_s], {\widetilde Z}_s, E[{\widetilde V}_s])\right)\, ds\right)^2\right]\\
\le\, &C^2{\widetilde E}_t\left[\left(\int_t^T\left(|Y_s-{\widetilde Y}_s|+|E[U_s-{\widetilde U}_s]|+(1+|E[V_s]|^\alpha+|E[{\widetilde V}_s]|^\alpha)|E[V_s-{\widetilde V}_s]|\right)\, ds\right)^2\right]\\
\le\, &3C^2(T-t)^2(|U-{\widetilde U}|^2_\infty+|Y-{\widetilde Y}|_\infty^2)\\
&+3C^2 \int_t^T (1+|E[V_s]|^{\alpha}+|E[{\widetilde V}_s]|^{\alpha})^2\, ds\int_t^T|E[V_s-{\widetilde V}_s]|^2\, ds\\
\le\, &3C^2(T-t)^2(|U-{\widetilde U}|^2_\infty+|Y-{\widetilde Y}|_\infty^2)\\
&+9C^2 \int_t^T (1+|E[V_s]|^{2\alpha}+|E[{\widetilde V}_s]|^{2\alpha})\, ds\int_t^T|E[V_s-{\widetilde V}_s]|^2\, ds
\end{split}
\end{eqnarray}
We have for $t\in [T-\varepsilon, T]$,
\begin{eqnarray}
    \begin{split}
&\ \int_t^T (1+|E[V_s]|^{2\alpha}+|E[{\widetilde V}_s]|^{2\alpha})\, ds\\
\le &\ \int_t^T (1+E[|V_s|^2]^{\alpha}+E[|{\widetilde V}_s|^2]^{\alpha})\, ds\\
\le & \varepsilon+\varepsilon^{1-\alpha}E\left[\int_t^T |V_s|^2ds\right]^\alpha+\varepsilon^{1-\alpha}E\left[\int_t^T |{\widetilde V}_s|^2ds\right]^\alpha\\
\le &\ \varepsilon^{1-\alpha} \left(T^\alpha+2+\alpha E\int_t^T |V_s|^2\, ds+\alpha E\int_t^T|{\widetilde V}_s|^2\, ds\right)\\
\le &\ \varepsilon^{1-\alpha} \left(T^\alpha+2+\alpha \|V\cdot W\|^2_{BMO_2(\mathbb P)}+\alpha \|{\widetilde V}\cdot W\|^2_{BMO_2(\mathbb P)}\right)\\
\le &\ \varepsilon^{1-\alpha} (T^\alpha+2+2\alpha A).
\end{split}
\end{eqnarray}
Concluding the above estimates, we have for $t\in [T-\varepsilon, T]$,
\begin{eqnarray}
    \begin{split}
  &|Y_t-{\widetilde Y}_t|^2+{\widetilde E}_t\left[\int_t^T|Z_s-{\widetilde Z}_s|^2\, ds\right]\\
\le\, &3C^2\varepsilon^2(|U-{\widetilde U}|^2_\infty+|Y-{\widetilde Y}|^2_\infty)\\
&+9C^2 \left(T^\alpha+2+2\alpha A\right)\varepsilon^{1-\alpha} \|(V-{\widetilde V})\cdot W\|^2_{BMO_2(\mathbb P)}.
\end{split}
\end{eqnarray}
In view of estimates~\eqref{uniform BMO estimate}, noting that $1-3C^2\varepsilon^2\ge {2 \over 3}$, we have for $t\in [T-\varepsilon, T]$,
\begin{eqnarray}
    \begin{split}
&|Y-{\widetilde Y}|^2_\infty+3c_1^2\|(Z-{\widetilde Z})\cdot W\|^2_{BMO_2(\mathbb P)}\\
\le \, &9C^2\ \varepsilon^2|U-{\widetilde U}|^2_\infty
+27C^2 \left(T^\alpha+2+2\alpha A\right)\varepsilon^{1-\alpha} \|(V-{\widetilde V})\cdot W\|^2_{BMO_2(\mathbb P)}.
\end{split}
\end{eqnarray}

It is then standard to show that there is a very small positive number $\varepsilon$ such that the quadratic solution map $\Gamma$ is a contraction on the previously given set $\mathscr{B}_\varepsilon$, by noting that $A\le 6\gamma^{-2}e^{\gamma |\xi|_\infty}$ from \eqref{A}. The proof is completed by choosing a sufficiently small $r_\varepsilon>0$ such that $\mathbb{B}_\varepsilon\subset \mathscr{B}_\varepsilon$.

\subsection{Global solution: the proof of Theorem~\ref{thm 2}}

Let us first note that there exists a constant $\widetilde C>0$ such that $|\xi|^2\le \widetilde C$ and
$$
 |2x h(s,y,\bar y, z, \bar z)+2x f(s,0,0,0,0)|\le \widetilde C+\widetilde C|x|^2+\widetilde C(|y|^2 +|\bar y|^2)
$$
for any $(x,y,\bar y, z, \bar z)\in R\times R^2\times (R^d)^2$.
Let $\alpha (\cdot)$ be the unique solution of the following ordinary differential equation:
$$
\alpha (t)=\widetilde C+\int_t^T\widetilde C\, ds+\int_t^T(2\widetilde C+\widetilde C)\alpha (s)\, ds, \quad t\in [0,T].
$$

It is easy to see that $\alpha (\cdot)$ is a continuous decreasing function and  we have
$$
\alpha (t)=\widetilde C+\int_t^T\widetilde C[1+2\alpha (s)]\, ds+\widetilde C\int_t^T\alpha (s)\, ds, \quad t\in [0,T].
$$

Define
$$
\lambda :=\sup_{t\in [0,T]}\alpha(t)=\alpha(0).
$$

As $|\xi|^2\le \widetilde C\le \lambda$, Theorem \ref{thm 1} shows that there exists $\eta_\lambda>0$ which only depends on $\lambda$, such that BSDE has a local solution $(Y, Z)$ on $[T-\eta_\lambda, T]$ and it can be constructed through the Picard iteration.

Consider the Picard iteration:
\begin{eqnarray}
Y^0_t&=&\xi+\int_t^TZ_s^0\, dW_s; \quad \mbox{ and for } \quad j\ge 0,\nonumber \\
 Y^{j+1}_t  &=& \xi+\int_t^T[f(s,0,0,Z_s^{j+1},0)-f(s,0,0,0,0)]\, ds\nonumber\\
 && +\int_t^T[f(s,0,0,0,0)+f(s, Y_s^j, E[Y_s^j], Z_s^j, E[Z_s^j])-f(s,0,0,Z_s^j,0)]\, ds\nonumber\\
 &&-\int_t^TZ_s^{j+1}\, dW_s \nonumber \\
   &=& \xi+\int_t^T [f(s,0,0,0,0)+h(s, Y_s^j,E[Y_s^j], Z_s^j, E[Z_s^j])]\, ds\nonumber\\
   && -\int_t^TZ_s^{j+1}\, d\widetilde W_s^{j+1}, \quad t\in [T-\eta_\lambda,T], \nonumber
\end{eqnarray}
where 
$$f(s,0,0,Z_s^{j+1},0)-f(s,0,0,0,0)=Z_s^{j+1}\beta_s^{j+1}$$ 
and the process
$$\widetilde W_t^{j+1}=W_t-\int_0^t \beta_s^{j+1}{\bf 1}_{[T-\eta_\lambda, T]}(s)ds$$ is a Brownian motion under an equivalent probability measure $P^{j+1}$ which we denote by $\widetilde P$ with loss of generality, and under which the expectation is denoted by $\widetilde E$.
Using It\^o's formula, it is straightforward to deduce the following estimate for  $r\in [T-\eta_\lambda,t]$,
\begin{eqnarray}\label{secondY}
   && {\widetilde E}_r[|Y_t^{j+1}|^2]+{\widetilde E}_r [\int_r^T|Z_s^{j+1}|^2\, ds]\nonumber \\
   &=& {\widetilde E}_r [|\xi|^2]+\widetilde C {\widetilde E}_r\int_t^T 2Y_s^{j+1}[f(s,0,0,0,0)+h(s, Y_s^j, E[Y_s^j], Z_s^j, E[Z_s^j])]\, ds\nonumber \\
   &\le& \widetilde C+\widetilde C\int_t^T{\widetilde E}_r[|Y_s^{j+1}|^2]\, ds+\widetilde C\int_t^T\{{\widetilde E}_r [|Y_s^j|^2]+|E[Y_s^j]|^2+1\}\, ds.
\end{eqnarray}

In what follows, we show by induction the following inequality:
\begin{equation}\label{induction}
|Y_t^{j}|^2\le \alpha(t), \quad t\in [T-\eta_\lambda,T].
\end{equation}

In fact, it is trivial to see that $|Y_t^0|^2\le \alpha(t)$, and let us suppose $|Y_t^j|^2\le \alpha(t)$ for $t\in [T-\eta_\lambda,T]$. Then, from \eqref{secondY},
$$
{\widetilde E}_r[|Y_t^{j+1}|^2]\le \widetilde C+\widetilde C\int_t^T[1+2\alpha(s)]\, ds+\widetilde C\int_t^T{\widetilde E}_r [|Y_s^{j+1}|^2]\, ds, \quad t\in [T-\eta_\lambda,T].
$$
From the comparison theorem, we have
$$
{\widetilde E}_r [|Y_t^{j+1}|^2]\le \alpha(t),  \quad t\in [T-\eta_\lambda,T].
$$
Setting $r=t$, we have
$$
|Y_t^{j+1}|^2\le \alpha(t),  \quad t\in [T-\eta_\lambda,T].
$$
Therefore, inequality~\eqref{induction} holds.

As $Y_t=\lim_j Y_t^j$ , our constructed local solution $(Y,Z)$ in $[T-\eta_\lambda,T]$   satisfies then the following estimate:
$$
|Y_t|^2\le \alpha(t), \quad t\in [T-\eta_\lambda, T], \quad {and} \quad |Y_t|^2\le \alpha(t), \quad t\in [T-\eta_\lambda, T]. $$

In particular, $|Y_{T-\eta_\lambda}|^2\le \alpha(T-\eta_\lambda)\le \lambda$.

Taking $T-\eta_\lambda$ as the terminal time and $Y_{T-\eta_\lambda}$ as terminal value, Theorem \ref{thm 1} shows that BSDE has a local solution $(Y, Z)$ on $[T-2\eta_\lambda, T-\eta_\lambda]$ through the Picard iteration. Once again, using the Picard iteration and the fact that $|Y_{T-\eta}|^2\le \alpha(T-\eta)$, we deduce that $|Y_t|^2\le \alpha(t)$, for $t\in [T-2\eta_\lambda, T-\eta_\lambda]$.
Repeating the preceding process, we can extend the pair $(Y, Z)$ to the whole interval $[0, T]$ within a finite steps such that $Y$ is uniformly bounded by $\lambda$. We now show that $Z\cdot W$ is a $BMO(P)$ martingale.

Identical to the proof of inequality~\eqref{quadratic estimate}, we have
 \begin{eqnarray}
    \begin{split}
      &\phi(Y_t)+{1\over2} E_t\left[\int_t^T|Z_s|^2\, ds\right] \\
      \le\, & \phi(|\xi|_\infty)+ CE_t\left[\int_t^T|\phi'(Y_s)|\left(1+|Y_s|+E[|Y_s|]\right)\, ds\right]\\
       \le\, & \phi(|\xi|_\infty)+ C\phi'(\lambda)E_t\left[\int_t^T\left(1+\lambda+\lambda\right)\, ds\right].
    \end{split}
\end{eqnarray}
Consequently, we have
\begin{eqnarray}
    \begin{split}
      & \|Z\cdot W\|^2_{BMO_2(\mathbb P)}=\sup_{\tau}E_\tau\left[\int_\tau^T|Z_s|^2\, ds\right] \\
            \le\, & 2\phi(|\xi|_\infty)+ 4C\phi'(\lambda)(1+\lambda)T.
    \end{split}
\end{eqnarray}

Finally, we prove the uniqueness. Let $(Y,Z)$ and $(\widetilde Y, \widetilde Z)$ be two adapted solutions. Then, we have (recall that $\beta$ is defined by (?))
\begin{eqnarray}
    \begin{split}
Y_t-{\widetilde Y}_t=\, & \int_t^T\left[h(s,Y_s,E[Y_s], \tilde{Z}_s, E[Z_s])-h(s, {\widetilde Y}_s,E[{\widetilde Y}_s], {\widetilde Z}_s, E[{\widetilde Z}_s])\right]\, ds \\
&-\int_t^T(Z_s-{\widetilde Z}_s) \, d{\widetilde W}_s, \quad t\in [0,T].
\end{split}
\end{eqnarray}
Similar to the first two inequalities in \eqref{square}, for any stopping time $\tau$ which takes values in $[T-\varepsilon, T]$, we have
\begin{eqnarray}
    \begin{split}
&|Y_\tau-{\widetilde Y}_\tau|^2+{\widetilde E}_{\tau}\left[\int_\tau^T|Z_s-{\widetilde Z}_s|^2\, ds\right]\\
= \, & {\widetilde E}_{\tau}\left[\left(\int_\tau^T\left[h(s,Y_s,E[Y_s], \tilde{Z}_s, E[Z_s])-h(s, {\widetilde Y}_s,E[{\widetilde Y}_s], {\widetilde Z}_s, E[{\widetilde Z}_s])\right]\, ds\right)^2\right]\\
\le\, &C^2{\widetilde E}_{\tau}\left[\left(\int_\tau^T\left[|Y_s-{\widetilde Y}_s|+|E[Y_s-{\widetilde Y}_s]|+|E[Z_s-{\widetilde Z}_s]|\right]\, ds\right)^2\right]\\
\le \, &6C^2\varepsilon^2|Y-\widetilde Y|_\infty^2+3C^2 \varepsilon {\widetilde E}_{\tau}\left[\int_\tau^T|Z_s-{\widetilde Z}_s|^2\, ds\right]+3C^2 \varepsilon E\left[\int_{T-\varepsilon}^T|Z_s-{\widetilde Z}_s|^2\, ds\right]\\
\le \, &6C^2\varepsilon^2|Y-\widetilde Y|_\infty^2+3C^2 \varepsilon \|(Z-{\widetilde Z})\cdot \widetilde W\|^2_{BMO_2(\widetilde P)}+3C^2 \varepsilon \|(Z-{\widetilde Z})\cdot  W\|^2_{BMO_2(\mathbb P)}\\
\le \, &6C^2\varepsilon^2|Y-\widetilde Y|_\infty^2+3C^2 (1+c_2^2)\varepsilon \|(Z-{\widetilde Z})\cdot  W\|^2_{BMO_2(\mathbb P)}.
\end{split}
\end{eqnarray}
Therefore, we have (on the interval $[T-\varepsilon, T]$)
\begin{eqnarray}
    \begin{split}
&\left|Y-{\widetilde Y}\right|^2_\infty+c_1^2\left\|(Z-{\widetilde Z})\cdot W\right\|^2_{BMO_2(\mathbb P)} \\
\le \, &6C^2\varepsilon^2\left|Y-\widetilde Y\right|_\infty^2+3C^2 (1+c_2^2)\varepsilon \left\|(Z-{\widetilde Z})\cdot  W\right\|^2_{BMO_2(\mathbb P)}.
\end{split}
\end{eqnarray}
Note that since
$$|\beta|\le C(1+|Z|+|\widetilde Z|), $$
the two generic constants $c_1$ and $c_2$ only depend on the sum
$$
\left\|Z\cdot  W\right\|^2_{BMO_2(\mathbb P)}+\left\|{\widetilde Z}\cdot  W\right\|_{BMO_2(\mathbb P)}.
$$

Then when $\varepsilon$ is sufficiently small, we conclude that  $Y=\widetilde Y$ and $Z=\widetilde Z$ on $[T-\varepsilon, T]$. Repeating iteratively with a finite of times, we have the uniqueness
on the given interval $[0, T]$.

\section{The expected term is additive and has a quadratic growth in the second unknown variable}
Let us first consider the following quadratic BSDE with mean term:
\begin{eqnarray}\label{simpleBSDE}
  Y_t &=& \xi+\int_t^T\left[f_1(s,Z_s)+E[f_2(s,Z_s,\mathbb E[Z_s])\right]\, ds-\int_t^TZ_s\, dB_s, \quad t\in [0,T]，
\end{eqnarray}
where $f_1:\Omega\times [0,T]\times \mathbb R^d\times \mathbb R^d\rightarrow \mathbb R$ satisfies:
$f_1(\cdot,z)$ is an adapted process for any $z$, and
\begin{equation}\label{quadraticgenerator1}
|f_1(t,0)|\le C,\quad |f_1(t,z)-f_i(t,z')|\le C(1+|z|+|z'|)|z-z'|)|\bar{z}-\bar{z}'|),
\end{equation}
and $f_2:\Omega\times [0,T]\times \mathbb  R^d\times \mathbb R^d\rightarrow \mathbb R$ satisfies:
 $f_2(\cdot,z,\bar{z})$ is an adapted process for any $z$ and $\bar{z}$, and
\begin{equation}\label{quadraticgenerator2}
|f_2(t,z,\bar{z})|\le C(1+|z|^2+|z'|^2).
\end{equation}

\begin{prop}
Let us suppose that $f_1$ and $f_2$ be two  generators satisfying the above conditions and $\xi$ be a bounded random variable.
 Then (\ref{simpleBSDE}) admits a unique solution $(Y,Z)$ such that $Y$ is bounded and $Z\cdot W$ is a BMO martingale.
\end{prop}

\begin{proof}
Let us first prove the existence. We solve this equation in two steps:

Step one. First solve the following BSDE:
\begin{equation}\label{bsdetilde}
    {\widetilde Y}_t=\xi+\int_t^Tf_1(s,Z_s)\, ds-\int_t^TZ_s\cdot dW_s.
\end{equation}
It is well known that this BSDE admits a unique solution $(\tilde{Y},Z)$ such that $\tilde{Y}$ is bounded and $Z\cdot W$ is a BMO martingale.

Step two. Define
\begin{equation}\label{}
    Y_t={\widetilde Y}_t+\int_t^TE[f_2(s,Z_s,\mathbb E[Z_s])]\, ds.
\end{equation}
Then
\begin{eqnarray}
  Y_t &=& \xi+\int_t^T\left[f_1(s,Z_s)+E[f_2(s,Z_s,\mathbb E[Z_s])\right]\, ds-\int_t^TZ_s\cdot dW_s, \quad t\in [0,T].
\end{eqnarray}

The uniqueness can be proved in a similar way:
Let $(Y^1,Z^1)$ and $(Y^2,Z^2)$ be two solutions. Then set
$$\tilde{Y}_t^1=Y_t^1-\int_t^T E[f_2(s,Z^1_s,\mathbb E[Z^1_s])]\, ds, \quad \tilde{Y}_t^2=Y^2_t-\int_t^T E[f_2(s,Z^2_s,\mathbb E[Z^2_s])]\, ds.$$
$(\tilde{Y}^1,Z^1)$ and $(\tilde{Y}^2,Z^2)$ being solution of the same BSDE (\ref{bsdetilde}), from the uniqueness of solution to this BSDE,
$$\tilde{Y}^1=\tilde{Y}^2,\quad Z^1=Z^2,$$
Hence $Y^1=Y^2$, and $Z^1=Z^2$.
\end{proof}

Now we consider a more general form of BSDE with a mean term:
\begin{eqnarray}\label{BSDEmean2}
Y_t& = &\xi+\int_t^T\left[f_1(s,Y_s,\mathbb E[Y_s],Z_s,\mathbb E[Z_s])+E[f_2(s,Y_s,\mathbb E[Y_s],Z_s,\mathbb E[Z_s])\right]\, ds\nonumber\\
& & -\int_t^TZ_s\cdot dW_s.
\end{eqnarray}
Here for $i=1,2$,  $f_i:\Omega\times [0,T]\times \mathbb R\times \mathbb R\times \mathbb R^d\times \mathbb R^d\rightarrow \mathbb R$ satisfies:
for any $(y,\bar{y},z,\bar{z})$, $f_i(t,y,\bar{y},z,\bar{z})$ is an adapted process, and
$$|f_1(t,y,\bar{y},0,\bar{z})|+|f_2(t,0,0,0,0)|\le C,$$
$$|f_1(t,y,z,\bar{y},\bar{z})-f_1(t,y',z',\bar{y}',\bar{z}')|\le C(|y-y'|+|\bar{y}-\bar{y}'|+|\bar{z}-\bar{z}'|+(1+|z|+|z'|)|z-z'|),$$
$$|f_2(t,y,z,\bar{y},\bar{z})-f_2(t,y',z',\bar{y}',\bar{z}')|\le C(|y-y'|+|\bar{y}-\bar{y}'|+(1+|z|+|z'|+|\bar{z}|+|\bar{z}'|)(|z-z'|+|\bar{z}-\bar{z}'|)).$$

We have the following result.

\begin{thm} Assume that $f_1$ and $f_2$ satisfy the above conditions, and $\xi$ is a bounded random variable.
BSDE (\ref{BSDEmean2}) has a unique solution $(Y,Z)$ such that $Y\in {\cal S}^\infty$ and
$Z\cdot W\in BMO_2(\mathbb P)$.
\end{thm}

\begin{exam} The condition on $f_1$ means that this function should be bounded with respect to
$(y,\bar{y},\bar{z})$. For example,
$$f_1(s,y,\bar{y},z,\bar{z})=1+|\sin(y)|+|\sin(\bar{y})|+\frac{1}{2}|z|^2+|\sin(\bar{z})|,$$
$$f_2(s,y,\bar{y},z,\bar{z})=1+|y|+|y'|+\frac{1}{2}(|z|+|\bar{z}|)^2.$$
\end{exam}
\begin{proof}
We prove the theorem by a fixed point argument. 
Let $U\in {\cal S}^{\infty}$, and $V\cdot W\in BMO_2(\mathbb P)$, we define 
$(Y,Z\cdot W)\in {\cal S}^\infty\times BMO_2(\mathbb P)$ as the unique solution to BSDE with mean:
\begin{eqnarray}\label{BSDEmean3}
Y_t& = &\xi+\int_t^T\left[f_1(s,U_s,\mathbb E[U_s],Z_s,\mathbb E[V_s])+E[f_2(s,U_s,\mathbb E[U_s],Z_s,\mathbb E[Z_s])\right]\, ds\nonumber\\
& & -\int_t^TZ_s\cdot dW_s.
\end{eqnarray}
And we define the map $\Gamma:(U,V)\mapsto \Gamma(U,V)=(Y,Z)$ on ${\cal S}^\infty\times BMO_2(\mathbb P)$.
Set 
$$\tilde{Y}_t=Y_t-\int_t^T E\left[f_2(s,U_s,\mathbb E[U_s],Z_s,\mathbb E[Z_s])\right]\, ds,$$
then $(\tilde{Y},Z)$ is the solution to
$$\tilde{Y}_t = \xi+\int_t^Tf_1(s,U_s,\mathbb E[U_s],Z_s,\mathbb E[V_s])\, ds-\int_t^TZ_s\, dB_s.$$
As $|f_1(t,y,\bar{y},0,\bar{z})|\le C$, we have
$$|f_1(t,y,\bar{y},z,\bar{z})|\le C+\frac{C}{2}|z|^2.$$
Applying Ito's formula to $\phi_C(|\tilde{Y}|)$,
\begin{equation}\label{phi}
\phi_C(|\tilde{Y}_t|)+\int_t^T\frac{1}{2}\phi_C^{''}(|\tilde{Y}_s|)|Z_s|^2ds\le\phi(\xi)+\int_t^T |\phi_C'(|\tilde{Y}_s|)|(C+\frac{C}{2}|Z_s|^2)ds-\int_0^T\phi'(|\tilde{Y}_s|)Z_sdW_s,
\end{equation}
we can prove that there exists a constant $K>0$ such that
$$||Z||_{BMO}\le K.$$

Let $(U^1,V^1)$ and $(U^2,V^2)$, and $(Y^1,Z^1)$ and $(Y^2,Z^2)$ be the corresponding solution,
and 
$$\tilde{Y}^1_t = \xi+\int_t^Tf_1(s,U^1_s,\mathbb E[U^1_s],Z^1_s,\mathbb E[V^1_s])\, ds-\int_t^TZ^1_s\, dB_s,$$
$$\tilde{Y}^2_t = \xi+\int_t^Tf_1(s,U^2_s,\mathbb E[U^2_s],Z^2_s,\mathbb E[V^2_s])\, ds-\int_t^TZ^2_s\, dB_s.$$
There exists a bounded adapted process $\beta$ such that
$$f_1(s,U^1_s,\mathbb E[U^1_s],Z^1_s,\mathbb E[V^1_s])-f_1(s,U^1_s,\mathbb E[U^1_s],Z^2_s,\mathbb E[V^1_s])=\beta_s (Z_s^1-Z_s^2).$$
Then $\widetilde{W}_t:=W_t-\int_0^t\beta_sds$ is a Brown motion under the equivalent probability measure $\tilde{\mathbb P}$ defined by
$$d\tilde{\mathbb P}={\cal E}_0^T(\beta\cdot W)_0^T d\mathbb P.$$
We have
\begin{eqnarray}
& &\Delta \tilde{Y}_t+\int_t^T\Delta Z_s \, d{\widetilde W}_s\\
&= &\int_t^T\left[f_1(s,U^1_s,\mathbb E[U^1_s],Z^2_s,\mathbb E[V^1_s])-f_1(s,U^2_s,\mathbb E[U^2_s],Z^2_s, \mathbb E[V^2_s])\right]ds, \quad t\in [0,T].\nonumber
\end{eqnarray}
 For any stopping time $\tau$ which takes values in $[T-\varepsilon, T]$, taking square and then the conditional expectation with respect to $\widetilde{\mathbb P}$ (denoted by $\widetilde{\mathbb E}_\tau$), 
 we have
\begin{eqnarray}
    \begin{split}
&|\Delta \tilde{Y}_\tau|^2+{\widetilde E}_{\tau}\int_\tau^T|\Delta Z|^2ds\\
=&{\widetilde E}_{\tau}\left(\int_t^T\left[f_1(s,U^1_s,\mathbb E[U^1_s],Z^2_s,\mathbb E[V^1_s])-f_1(s,U^2_s,\mathbb E[U^2_s],Z^2_s,\mathbb E[V^2_s])\right]\, ds\right)^2\\
\le \, & C^2{\widetilde E}_{\tau}\left(\left[\int_\tau^T\left(|\Delta U_s|+
|\mathbb E[\Delta U_s]|+|\mathbb E[\Delta V_s]|\right)ds\right]^2\right)\\
\le \, & C^2\varepsilon( |\Delta U|_\infty^2+|\Delta V\cdot W|^2_{BMO_2(\mathbb P)}).
\end{split}
\end{eqnarray}
Therefore, we have (on the interval $[T-\varepsilon, T]$)
\begin{equation}
 \left|\Delta \tilde{Y}\right|^2_\infty+c_1^2\left\|(\Delta Z)\cdot W\right\|^2_{BMO_2(\mathbb P)} 
\le \, C^2\varepsilon( |\Delta U|_\infty^2+|\Delta V\cdot W|^2_{BMO_2(\mathbb P)}).
\end{equation}
Note that since
$$|\beta|\le C(1+|Z_1|+|Z_2), $$
the  generic constant $c_1$ and $c_2$ only depend on $C$ and $K$.

As
$$\Delta Y_t=\Delta \tilde{Y}_t+\int_t^T E\left[f_2(s,U^1_s,\mathbb E[U^1_s],Z^1_s,\mathbb E[Z^2_s])-f_2(s,U^2_s,\mathbb E[U^2_s],Z^2_s,\mathbb E[Z^2_s])\right]ds,$$
then
\begin{eqnarray*}|\Delta Y_t|&\le& |\Delta \tilde{Y}_t|+CE\left[\int_t^T \left(|\Delta U_s|+|\mathbb E\Delta U_s|\right)ds\right]\\
& &+CE\left[\int_t^T (1+|Z_s^1|+|Z_s^2|+|\mathbb E Z_s^1|+|\mathbb E Z_s^2|)\left(|\Delta Z_s|+|\mathbb E\Delta Z_s|\right)ds\right]
\end{eqnarray*}
we deduce that
\begin{eqnarray*}
|\Delta Y_t|^2&\le& C^2\Big(|\Delta \tilde{Y}_t|^2+E\left[\int_t^T \left(|\Delta U_s|+|\mathbb E\Delta U_s|\right)ds\right]^2\\
& &E\left[\int_t^T (1+|Z_s^1|+|Z_s^2|+|\mathbb E [Z_s^1]|+|\mathbb E [Z_s^2]|)\left(|\Delta Z_s|+|\mathbb E\Delta Z_s|\right)ds\right]^2\Big),
\end{eqnarray*}
which implies that
\begin{eqnarray*}
|\Delta Y|^2_\infty
&\le& \, C^2(|\Delta \tilde{Y}|_\infty^2+ \varepsilon |\Delta U|^2_\infty+  |(\Delta Z)\cdot W|^2_{BMO_2(\mathbb P)} )\\
&\le& \,C^2\varepsilon\left( |\Delta U|^2_\infty+||V\cdot W||_{BMP_2(P)}^2\right). 
\end{eqnarray*}

Then when $\varepsilon$ is sufficiently small, we conclude that the application is contracting on $[T-\varepsilon, T]$. Repeating iteratively with a finite of times, we have the existence and uniqueness
on the given interval $[0, T]$.
\end{proof}

\section{Multi-dimensional Case}

In this section, we will study the multi-dimensional case of (\ref{BSDEmean2}), where $f_1$ is Lipschitz.
Our result generalizes the corresponding one of Cheridito and Nam \cite{CheriditoNam2}.

We first consider the following  BSDE with mean term:
\begin{eqnarray}\label{simpleBSDElast}
  Y_t &=& \xi+\int_t^T\left[f_1(s,Z_s,\mathbb E[Z_s])+E[f_2(s,Z_s,\mathbb E[Z_s])\right]\, ds-\int_t^TZ_s\cdot dW_s, \quad t\in [0,T]，
\end{eqnarray}
where $f_1:\Omega\times [0,T]\times R^{d\times n}\times \mathbb R^{d\times n}\rightarrow R^n$ satisfies:
for any $z$ and $\bar{z}$, $f_1(t,z,\bar{z})$ is an adapted process, and
\begin{equation}\label{quadraticgenerator1last}
|f_i(t,0,0)|\le C,\quad |f_1(t,z,\bar{z})-f_i(t,z',\bar{z}')|\le C(|z-z'|+|\bar{z}-\bar{z}'|),
\end{equation}
and $f_2:\Omega\times [0,T]\times R^{d\times n} \times R^{d\times n}\rightarrow R^n$ satisfies:
for any $z,\bar{z}$, $f_2(t,z,\bar{z})$ is an adapted process, and
\begin{equation}\label{quadraticgenerator2}
|f_2(t,0,0)|\le C,\quad |f_2(t,z,\bar{z})-f_2(t,z',\bar{z}')|\le C(1+|z|+|z'|+|\bar{z}|+|\bar{z}'|)(|z-z'|+|\bar{z}-\bar{z}'|).
\end{equation}

\begin{prop}
Let us suppose that $f_1$ and $f_2$ be two  generators satisfying the above conditions and $\xi
\in L^2({\cal F}_T)$.
 Then (\ref{simpleBSDElast}) admits a unique solution $(Y,Z)$ such that $Y\in {\cal S}^2$ and $Z\in {\cal M}^2$.
\end{prop}

\begin{proof}
Let us first prove the existence. Our proof is divided into the following two steps.

Step one. First consider the following BSDE
\begin{equation}\label{bsdetilde}
    {\widetilde Y}_t=\xi+\int_t^Tf_1(s,Z_s,\mathbb E[Z_s])\, ds-\int_t^TZ_s\cdot dW_s, \quad t\in [0,T].
\end{equation}
It admits a unique solution $(\tilde{Y},Z)$ such that $\tilde{Y}\in {\cal S}^2$ and $Z\in {\cal M}^2$, see Buckdahn et al. \cite{BuckdahnLiPeng}.

Step two. Define
\begin{equation}\label{}
    Y_t={\widetilde Y}_t+\int_t^TE[f_2(s,Z_s,\mathbb E[Z_s])]\, ds,\quad t\in [0,T].
\end{equation}
Then, we have for $t\in [0,T]$,
\begin{eqnarray}
  Y_t &=& \xi+\int_t^T\left[f_1(s,Z_s,\mathbb E[Z_s])+E[f_2(s,Z_s,\mathbb E[Z_s])\right]\, ds-\int_t^TZ_s\cdot dW_s.
\end{eqnarray}

The uniqueness can be proved in a similar way:
Let $(Y^1,Z^1)$ and $(Y^2,Z^2)$ be two solutions. Then set
$$\tilde{Y}_t^1=Y_t^1-\int_t^T E[f_2(s,Z^1_s,\mathbb E[Z^1_s])]\, ds, \quad \tilde{Y}_t^2=Y^2_t-\int_t^T E[f_2(s,Z^2_s,\mathbb E[Z^2_s])]\, ds.$$
$(\tilde{Y}^1,Z^1)$ and $(\tilde{Y}^2,Z^2)$ being solution of the same BSDE (\ref{bsdetilde}), from the uniqueness of solution to this BSDE,
$$\tilde{Y}^1=\tilde{Y}^2,\quad Z^1=Z^2,$$
Hence $Y^1=Y^2$, and $Z^1=Z^2$.
\end{proof}

Now we consider the following more general form of BSDE with a mean term:
\begin{eqnarray}\label{BSDEmean2last}
Y_t& = &\xi+\int_t^T\left[f_1(s,Y_s,\mathbb E[Y_s],Z_s,\mathbb E[Z_s])+E[f_2(s,Y_s,\mathbb E[Y_s],Z_s,\mathbb E[Z_s])\right]\, ds\nonumber\\
& & -\int_t^TZ_s\cdot dW_s, t\in [0,T].
\end{eqnarray}
Here for $i=1,2$,  $f_i:\Omega\times [0,T]\times \mathbb R^n\times \mathbb R^n\times \mathbb R^{d\times n}\times \mathbb R^{d\times n}\rightarrow \mathbb R^n$ satisfies:
for any $(y,\bar{y},z,\bar{z})$, $f_i(t,y,\bar{y},z,\bar{z})$ is an adapted process, and
\begin{equation}\label{xxx1}
|f_1(t,y,\bar{y},0,0)|+|f_2(t,0,0,0,0)|\le C,
\end{equation}
\begin{equation}\label{xxx2}
|f_1(t,y,z,\bar{y},\bar{z})-f_1(t,y',z',\bar{y}',\bar{z}')|\le C(|y-y'|+|\bar{y}-\bar{y}'|+|z-z'|+|\bar{z}-\bar{z}'|),
\end{equation}
\begin{equation}\label{xxx3}
|f_2(t,y,z,\bar{y},\bar{z})-f_2(t,y',z',\bar{y}',\bar{z}')|\le C(|y-y'|+|\bar{y}-\bar{y}'|+(1+|z|+|z'|+|\bar{z}|+|\bar{z}'|)(|z-z'|+|\bar{z}-\bar{z}'|)).
\end{equation}

We have the following result.

\begin{thm} Assume that $f_1$ and $f_2$ satisfy the inequalities (\ref{xxx1})-(\ref{xxx3}), and $\xi\in L^2({\cal F}_T)$. Then
BSDE (\ref{BSDEmean2last}) has a unique solution $(Y,Z)$ such that $Y\in {\cal S}^2$ and
$Z\in {\cal M}^2$.
\end{thm}

\begin{exam} The condition (\ref{xxx1}) requires that both functions $f_1$ and $f_2$ should be bounded with respect to
$(y,\bar{y})$. For example,
$$f_1(s,y,\bar{y},z,\bar{z})=1+|\sin(y)|+|\sin(\bar{y})|+|z|+|\bar{z}|,$$
$$f_2(s,y,\bar{y},z,\bar{z})=1+|y|+|y'|+\frac{1}{2}(|z|+|\bar{z}|)^2.$$
\end{exam}
\begin{proof}
We prove the theorem by a fixed point argument. 
Let $U\in {\cal S}^{2}$, and $V\in{\cal M}^2$, we define 
$(Y,Z\cdot W)\in {\cal S}^2\times {\cal M}^2$ as the unique solution to BSDE with mean:
\begin{eqnarray}\label{BSDEmean3last}
Y_t& = &\xi+\int_t^T\left[f_1(s,U_s,\mathbb E[U_s],Z_s,\mathbb E[Z_s])+E[f_2(s,U_s,\mathbb E[U_s],Z_s,\mathbb E[Z_s])\right]\, ds\nonumber\\
& & -\int_t^TZ_s\cdot dW_s.
\end{eqnarray}
We define the map $\Gamma:(U,V)\mapsto \Gamma(U,V)=(Y,Z)$ on ${\cal S}^2\times {\cal M}^2$.
Set 
$$\tilde{Y}_t=Y_t-\int_t^T E\left[f_2(s,U_s,\mathbb E[U_s],Z_s,\mathbb E[Z_s])\right]\, ds,\quad t\in [0,T].$$
Then $(\tilde{Y},Z)$ is the solution to
$$\tilde{Y}_t = \xi+\int_t^Tf_1(s,U_s,\mathbb E[U_s],Z_s,\mathbb E[V_s])\, ds-\int_t^TZ_s\cdot dW_s,  t\in [0,T].$$
As $|f_1(t,y,\bar{y},0,0)|\le C$, we have
$$|f_1(t,y,\bar{y},z,\bar{z})|\le C(1+|z|+|\bar{z}|).$$
Applying Ito's formula to $|\tilde{Y}|^2$, we have
\begin{equation}\label{philast}
|\tilde{Y}_t|^2+\int_t^T|Z_s|^2ds\le|\xi|^2+2C\int_t^T |\tilde{Y}_s|(C+|Z_s|+|\mathbb E[ Z]|)ds-2\int_t^T(Y_s,Z_s\cdot dW_s).
\end{equation}
Further, using standard techniques, we can prove that there exists a constant $K>0$ such that
$$||Z||_{{\cal M}^2}\le K.$$

For $(U^i,V^i)\in {\cal S}^2\times {\cal M}^2$, define $(Y^i,Z^i)=\Gamma(U^i,V^i)$ with $i=1,2$. Further, set for $t\in [0,T]$, 
$$\tilde{Y}^1_t: = \xi+\int_t^Tf_1(s,U^1_s,\mathbb E[U^1_s],Z^1_s,\mathbb E[V^1_s])\, ds-\int_t^TZ^1_s\cdot  dW_s,$$
$$\tilde{Y}^2_t: = \xi+\int_t^Tf_1(s,U^2_s,\mathbb E[U^2_s],Z^2_s,\mathbb E[V^2_s])\, ds-\int_t^TZ^2_s\cdot dW_s.$$

We have
\begin{eqnarray}
& &\Delta \tilde{Y}_t+\int_t^T\Delta Z_s \, d{ W}_s\\
&= &\int_t^T\left[f_1(s,U^1_s,\mathbb E[U^1_s],Z^1_s,\mathbb E[Z^1_s])-f_1(s,U^2_s,\mathbb E[U^2_s],Z^2_s, \mathbb E[Z^2_s])\right]ds, \quad t\in [0,T].\nonumber
\end{eqnarray}
 For any $t\in [T-\varepsilon, T]$, taking square and then  expectations on both sides of the last equality, 
 we have
\begin{eqnarray}
    \begin{split}
&\mathbb E[|\Delta \tilde{Y}_t|^2]+{\mathbb  E}\left[\int_t^T|\Delta Z_s|^2ds\right]\\
=&{\mathbb E}_{\tau}\left(\int_t^T\left[f_1(s,U^1_s,\mathbb E[U^1_s],Z^1_s,\mathbb E[Z^1_s])-f_1(s,U^2_s,\mathbb E[U^2_s],Z^2_s,\mathbb E[Z^2_s])\right]\, ds\right)^2\\
\le \, & C^2{\mathbb E}\left(\left[\int_\tau^T\left(|\Delta U_s|+
|\mathbb E[\Delta U_s]|+|\Delta Z_s|+
|\mathbb E[\Delta Z_s]|\right)ds\right]^2\right)\\
\le \, & C^2\varepsilon( |\Delta U|_{{\cal S}^2}+|\Delta Z|^2_{{\cal M}^2}).
\end{split}
\end{eqnarray}
Therefore, we have (on the interval $[T-\varepsilon, T]$) for a sufficiently small $\varepsilon>0$,
\begin{equation}\label{4.12}
 \left|\Delta \tilde{Y}\right|_{{\cal S}^2}^2+\left\|\Delta Z\right\|^2_{{\cal M}^2} 
\le \, C^2\varepsilon( |\Delta U|_{{\cal S}^2}^2).
\end{equation}

As
$$\Delta Y_t=\Delta \tilde{Y}_t+\int_t^T E\left[f_2(s,U^1_s,\mathbb E[U^1_s],Z^1_s,\mathbb E[Z^1_s])-f_2(s,U^2_s,\mathbb E[U^2_s],Z^2_s,\mathbb E[Z^2_s])\right]ds,$$
we have
\begin{eqnarray*}|\Delta Y_t|&\le& |\Delta \tilde{Y}_t|+CE\left[\int_t^T \left(|\Delta U_s|+|\mathbb E\Delta U_s|\right)ds\right]\\
& &+CE\left[\int_t^T (1+|Z_s^1|+|Z_s^2|+|\mathbb E Z_s^1|+|\mathbb E Z_s^2|)\left(|\Delta Z_s|+|\mathbb E\Delta Z_s|\right)ds\right].
\end{eqnarray*}
Moreover, we deduce that 
\begin{eqnarray*}
|\Delta Y_t|^2&\le& C^2\biggl(|\Delta \tilde{Y}_t|^2+E\left[\int_t^T \left(|\Delta U_s|+|\mathbb E\Delta U_s|\right)ds\right]^2\\
& &+E\left[\int_t^T (1+|Z_s^1|+|Z_s^2|+|\mathbb E Z_s^1|+|\mathbb E Z_s^2|)\left(|\Delta Z_s|+|\mathbb E\Delta Z_s|\right)ds\right]^2\biggl).
\end{eqnarray*}
In view of (\ref{4.12})
\begin{eqnarray*}
\\\left|\Delta \tilde{Y}\right|_{{\cal S}^2}^2+\left\|\Delta Z\right\|^2_{{\cal M}^2} 
&\le& \,C^2\varepsilon\left(|\Delta U|_{{\cal S}^2}^2 \right). 
\end{eqnarray*}

Then when $\varepsilon$ is sufficiently small, we conclude that the application is contracting on $[T-\varepsilon, T]$. Repeating iteratively with a finite of times, we have the existence and uniqueness
on the given interval $[0, T]$.

\end{proof}

\bibliographystyle{siam}

\end{document}